\documentclass[a4paper,12pt]{amsart} 
\usepackage{amssymb,euscript,verbatim,ulem,amsfonts}
\usepackage[utf8]{inputenc}      % pour les accents
\usepackage[T1]{fontenc}

\usepackage[polish,english]{babel}
\usepackage{tikz, hyperref, aliascnt, color, mathtools}
\usepackage{enumerate,comment,float}

\theoremstyle{theorem}
\newtheorem{theorem}{Theorem}
\newtheorem{corollary}[theorem]{Corollary}
\newtheorem{proposition}[theorem]{Proposition}

\theoremstyle{definition}
\newtheorem*{definition}{Definition}

\usepackage[utf8]{inputenc}      % pour les accents
\usepackage[T1]{fontenc}
\DeclareMathSymbol{\shortminus}{\mathbin}{AMSa}{"39}

\def\w{\omega}

\def\SS1{\mathbb{S}^1}

\def\Z{\mathbb{Z}} 
\def\R{\mathbb{R}}

\def\d{\delta}
\def\a{{\bf a}}

\def\w{{\bf w}}
\def\d{{\bf d}}
\def\a{{\bf a}}
\def\h{{\bf h}}

\def\C{\mathcal{C}}

\begin{document}

\title{The horizontal chord set  }
\markright{Horizontal chords}
%\author{Author names here}

\author{Diana Davis}

\author{Serge Troubetzkoy}
\maketitle 

\begin{abstract}
    We study the set of lengths of the horizontal chords of a continuous function. We give a new proof of Hopf's characterization of this set, and show that it implies that
     no matter which function we choose, at least half of the possible lengths occur.  
     We prove several results about functions for which \emph{all} the possible lengths  occur.
\end{abstract}

%\address[D.\ Davis]{Phillips Exeter Academy, Exeter, NH, USA.
%postal address: 20 Main Street, Exeter NH 03833}
%\email{ddavis@exeter.edu}

%\address[S.\ Troubetzkoy]{Aix Marseille Univ, CNRS, Centrale Marseille, I2M, Marseille, France.
%postal address: I2M, Luminy, Case 907, F-13288 Marseille Cedex 9, France}
%\email{serge.troubetzkoy@univ-amu.fr}

%\date{\today}

\section{Story}
The mathematical conference center CIRM, on the outskirts of Marseille in France, was created in 1981 by the French mathematics community and hosts over 4500 mathematicians each year. CIRM is situated on an estate with a thousand-year history, at the edge of the Calanques National Park.

On a beautiful summer day, two mathematicians walk from the CIRM research institute down, down, down to the Mediterranean Sea at the beautiful Calanque de Sugiton. After briefly appreciating the cliffs, the waves, and the people jumping off of the former into the latter, they then walk back up, up, up via the same path, the whole trip taking one hour. They wonder: is there a spot on the route that they passed by two times, exactly 23 minutes apart? How about $\ell$ hours apart, for any $\ell\in[0,1]$?

In this paper, we show that the answer to the above question is ``yes.'' Such a spot exists for every $\ell$ even if the mathematicians wander back and forth on the trail along the way (left and middle pictures in Figure \ref{fig:hikes}). On the other hand, if the mathematicians pass by CIRM on their return and walk in the other direction, and then go back to  part of the way to the Calanque before returning to CIRM (right picture in Figure \ref{fig:hikes}), there is no such guarantee; however, at least half of the possible times must be achieved. The main tool in the proof of this last result is Hopf's characterization of the possible set of lengths $\ell$ which can occur, and we also give a new proof of Hopf's result.

\begin{figure}[ht]
    \centering
\includegraphics[width=0.8\textwidth]{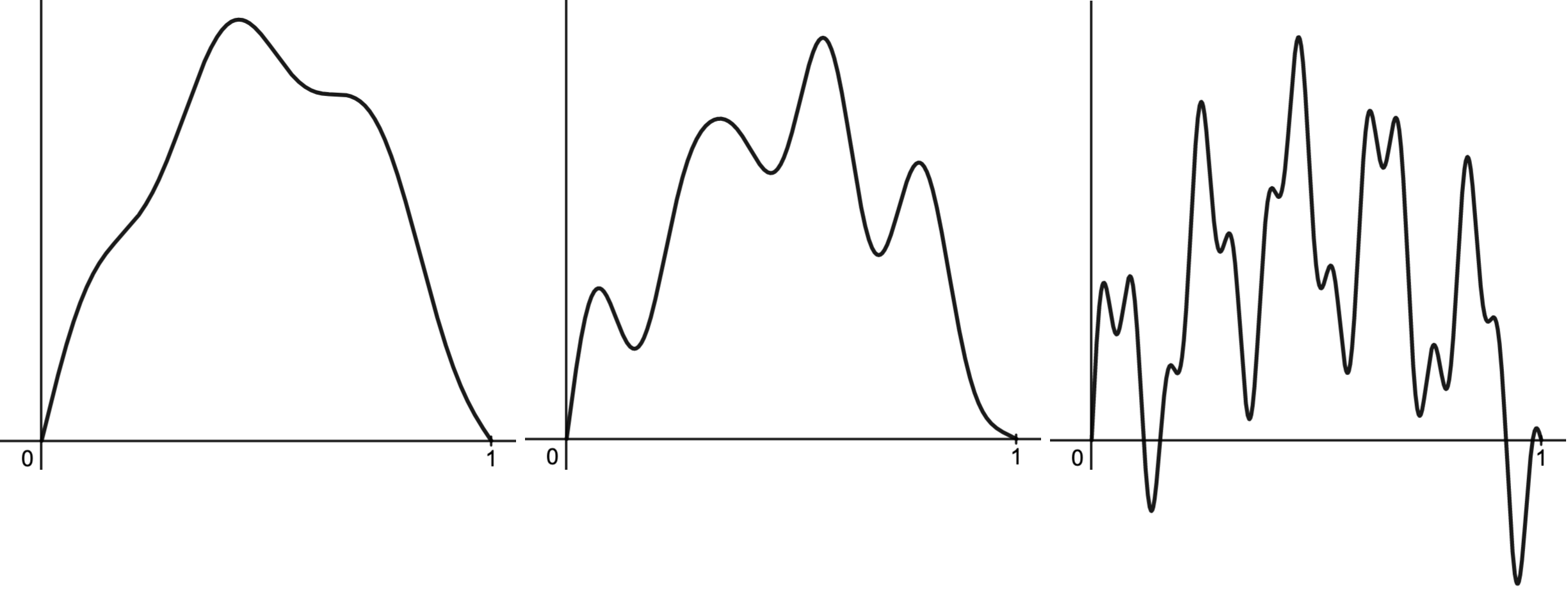}
    \caption{A simple hike, a meandering hike, and a wandering hike. Time is on the horizontal axis, and distance from CIRM towards the Calanque is on the vertical axis.}
    \label{fig:hikes}
\end{figure}

A different application of the issues discussed in this article, to average pace over time, was given in \cite{BDD}.

\section{Definitions}\label{sec:definitions}

By shifting and scaling, we may assume that length (in time) of the hike is $1$, and that the common value of $f(0)$ and $f(1)$ is $0$.
Let 
$$\C := \{ f: [0,1] \to \R, f  \text{  is continuous such that } f(0) = f(1)=0 \}.$$ 
The {\it horizontal chord set} is the set of
lengths horizontally connecting two points on the graph, i.e.,
$$S(f) := \{\ell \in [0,1]: \text{ there exists } 
s \in [0,1] \text{ with } f(s) = f(s + \ell)\}.$$
Hopf has characterized which sets $S \subset \R$ can be a horizontal chord set for some $f \in \C$ (\cite{H}, see \cite{BDD} for an English version of Hopf's proof).
We study the size and structure of $S(f)$. We begin with the following question: for which
functions $f \in \C$ does $S(f) = [0,1]$? 
We call this the {\it full chord property.}

\begin{definition}\label{def:mountain}
Suppose $f$ is a continuous function, $s_1 < s_2$, $a > 0$, $f(s_1)=f(s_2) = 0$ and $f((s_1,s_2)) = (0,a]$. Then we call $f|_{[s_1,s_2]}$ a {\it mountain} with {\it endpoints} $s_1$ and $s_2$. Fix $t\in[s_1, s_2]$ so that $f(t)=a$. (Note that while there may be several choices for $t$, none of our results depend on the choice.) Then we call $f|_{[s_1,t]}$ the {\it ascent}, $f|_{[t,s_2]}$ the {\it descent}, $a$ the {\it height}, and $|s_2-s_1|$ the {\it width}.

A  {\it valley} is exactly analogous to the above, with $a<0$ and ``ascent'' and ``descent'' exchanged.

A {\it mountain range} is a union of contiguous mountains, and possibly also intervals with $f(x)=0$ (including intervals degenerate to a point), i.e., essentially the same definition as above but with $f([s_1,s_2]) = [0,a]$. The {\it height} of a mountain range is the height of its tallest mountain, which exists because a continuous function has a maximum on a compact set. The ascent, descent, and width are as above.
We define a {\it valley range} analogously.

Note that mountain ranges and valley ranges can include horizontal intervals satisfying $f(x)=0$. Such a horizontal piece between a mountain range and a valley range can be considered as part of either one; the choice does not affect our results.

When we refer to a  mountain or a valley range we will always implicitly assume it is a  {\it maximal} mountain or valley range, modulo the choice of assignment of horizontal pieces at $f(x)=0$ as described above.
\end{definition}

\section{There and back again}

\begin{theorem}\label{mountainsarefull}
If $f$ is continuous and $f|_{[s_1,s_2]}$ is a mountain range, then $[0,s_2 - s_1] \subset S(f)$.
\end{theorem}

\begin{corollary}
If $f(x) \ge 0$ for all $x \in [0,1]$ 
then $f$ has the full chord property.
\end{corollary}
In particular the walk to Sugiton  has the full chord property.
Both of these results also hold for valley ranges.

\begin{proof}
Fix a mountain range $f|_{[s_1,s_2]}$, and consider the shifted mountain range \mbox{$f(x-\ell)$} (see Figure \ref{fig:shift}). For any shift $\ell\in[0,s_2-s_1]$,
exactly one endpoint of the mountain range is inside the other copy. 
Thus it follows from the Intermediate Value Theorem that the two mountain ranges intersect, so $f(x)=f(x-\ell)$ and there is a horizontal chord of \nopagebreak length $\ell$.
\end{proof}

A closely related problem is the following:
two mountain climbers begin at sea level at opposite ends of a mountain range, can they find routes along which to travel, always maintaining equal altitudes, until they eventually meet? A positive solution
implies the full chord property.
It holds for piecewise monotone mountain ranges \cite{W,GPY} and  for mountain ranges without plateaus \cite{K}, however this stronger property does not hold for certain mountains with plateaus \cite{K}.

\begin{figure}[ht]
    \centering
    \includegraphics[width=0.3\textwidth]{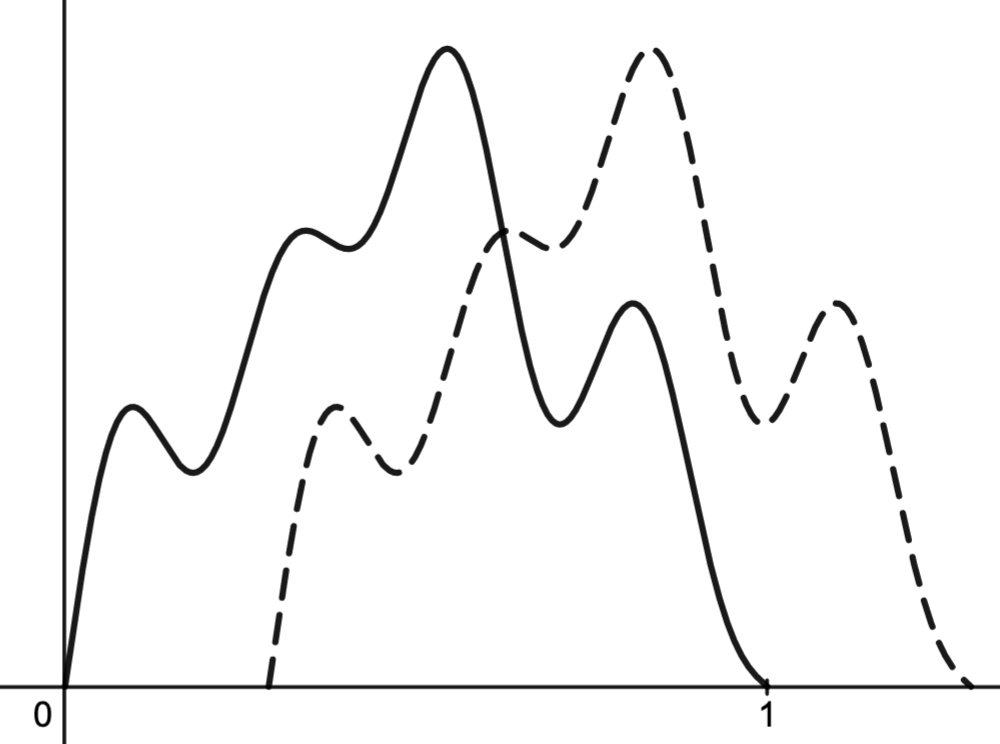}
    \caption{A mountain range and its shift always intersect.}
    \label{fig:shift}
\end{figure}

\section{Structure of the chord set}
Next we show that the chord set might not be the entire interval $[0,1]$:
\begin{theorem}\label{thm:1}
If $f$ has a mountain at one endpoint and a valley at the other endpoint, then $f$ does {not} have the full chord property.
\end{theorem} 
\begin{proof}
Let $w_1, w_2$ be the width of the mountain and the valley, respectively. Then any chord with horizontal displacement in the range $[1-\min\{w_1, w_2\}, 1)$ must have one endpoint in the mountain and one endpoint in the valley. Thus the function values at these two endpoints have different signs, so the chord between them is not horizontal. 
\end{proof}
Most of this paper consists of exploring closed intervals of chord lengths. However, the following figure shows that $S(f)$ does not necessarily consist of unions of closed intervals; there is a possibility that it can include isolated points.

By making lots of bumps, we can make the set of such additional chord lengths into a set with accumulation points.
\vspace{-0.3cm}
\begin{figure}[ht]
    \centering
    \includegraphics[height=0.3
    \textwidth]{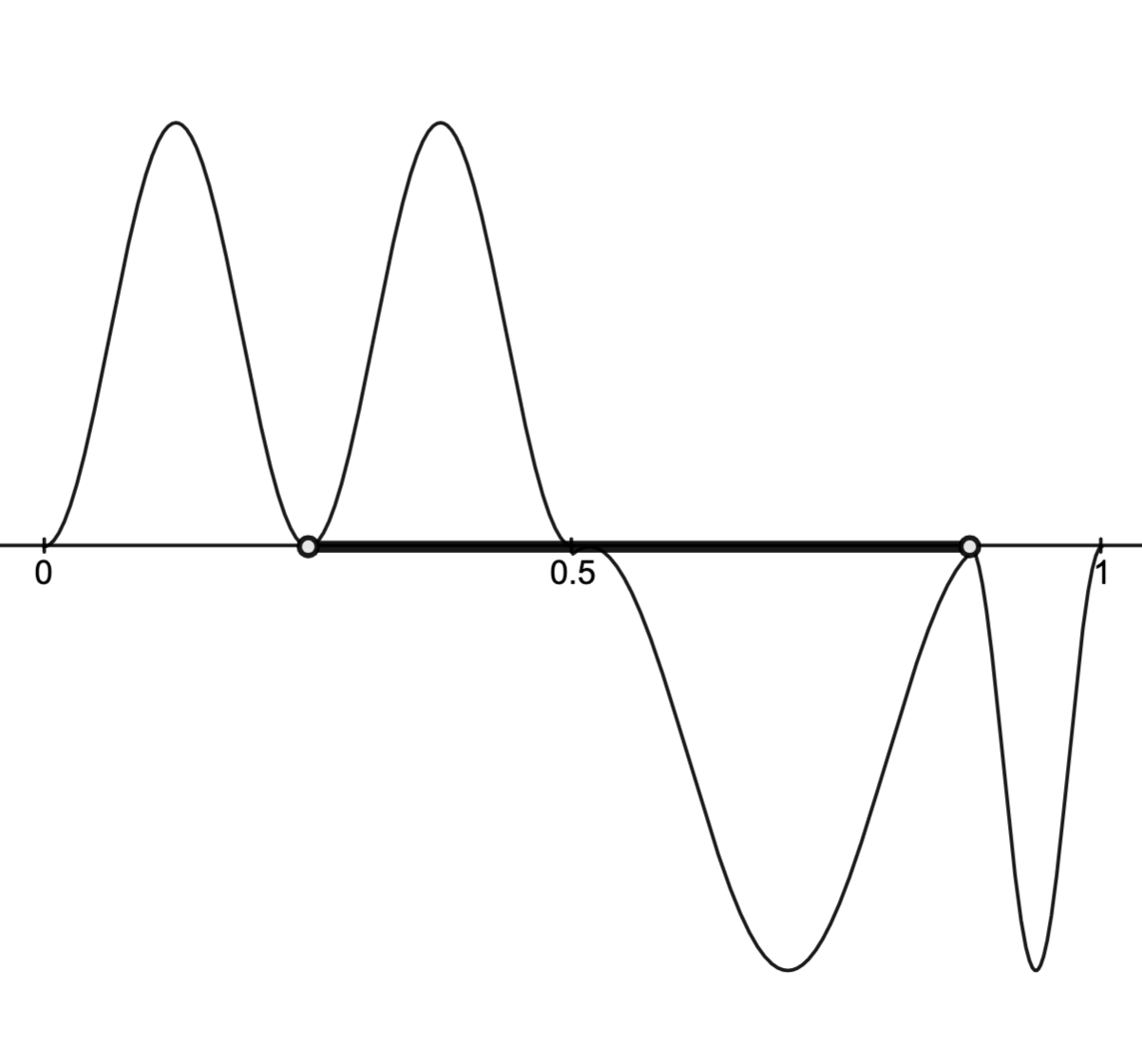} 
    \hspace{0.25in}
    \includegraphics[height=0.3\textwidth]{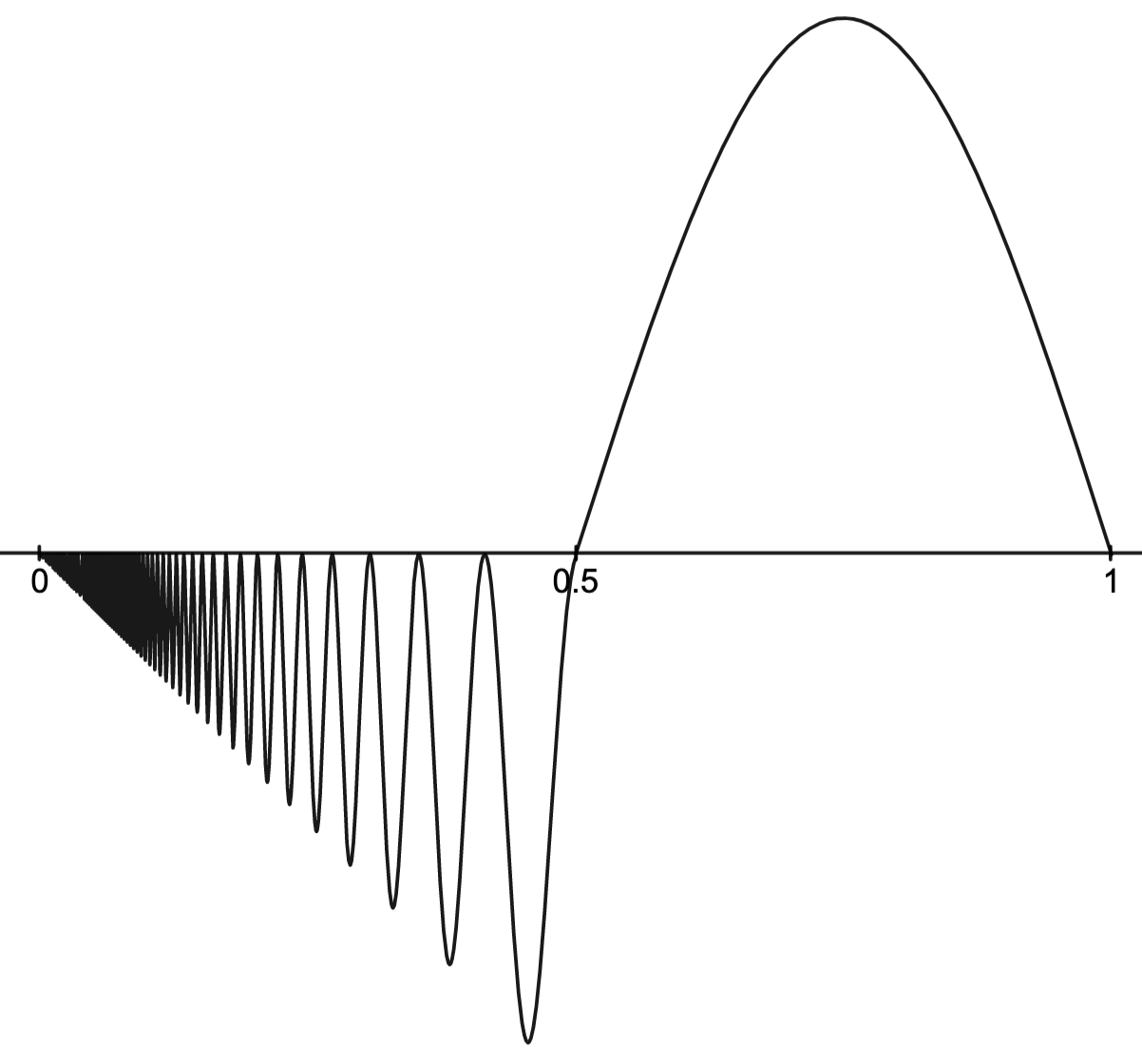}
    \vspace{0.15in}
    \caption{(left) A function with $S(f) = [0,1/2] \cup \{5/8, 3/4, 7/8,1\}$; the chord with length $5/8$ is marked in bold, and cannot be slid up nor down; (right) a function where $S(f)$ has an accumulation point at 1. }
    \label{fig:isolated}
\end{figure}

\section{Hopf's theorem}

We begin this section by the following simple result about continuous periodic functions.

\begin{proposition}\label{prop:prop}
Let $F:\mathbb{R} \to\mathbb{R}$ be a continuous periodic function with period $T$. Then for each $a \in \R$ the graph of $F(x - a)$ intersects the graph of $F(x)$ at least once in each interval of length $T$.
\end{proposition}

\begin{proof}
If $F$ is constant then there is nothing to prove, so suppose that $F$ is non-constant.
 Fix $a \in \R$,  
if $a = 0$ there is also nothing to prove, so assume $a \ne 0$.  

Since $[0,T]$ is compact $F|_{[0,T]}$ attains a miminum and a maximum, and since $F$ is periodic these are global extrema as well. Let $(x_m,y_m)$ be a minimum of $F$ and $(x_M,y_M)$ a maximum of $F$. 
If the result is not true then either  $F(x - a) > F(x)$ for all $x \in \R$ or the opposite inequality holds.  In the first case we get a contradiction by choosing $x = x_M$  and in the second case by choosing $x = x_m$.

If $F(x - a) = F(x)$ then $F(x -a + kT) = F(x)$ for all $k \in \Z$ by periodicity. 
\end{proof}

This proposition immediately yields a new (and simpler) proof of one direction of Hopf's classical theorem.
\begin{theorem} Let $f:[0,1]\to\R$ be a continuous function with $f(0)=f(1)=0$, and let $S(f)$ denote its horizontal chord length set.
Then the set $S^c  := \R^+ \setminus S(f)$
is open and additive; i.e., if $a,b \in S^c$ then \mbox{$a + b \in S^c$}.
\end{theorem}

\begin{proof}
Suppose $\ell_i \in S(f)$ and $\ell_i \to \ell$.  Then there are points $x_i \in [0,1]$ such that $x_i - \ell_i \in [0,1]$ and
 $f(x_i) = f(x_i -\ell_i)$.  By passing to a subsequence
we can assume that  $x_i$ converges to a point $x$ and thus $x_i -\ell_i \to x - \ell$ and by continuity of $f$ we have $f(x) = f(x -\ell)$.  Thus
$S(f)$ is closed and so $S^c$ is open.

For the proof we also consider the signed chord length set $$\bar{S}(f) :=\{ \ell \in \R: \text{ there exists } s \in [0,1] 
 \text{ such that } f(s -\ell) = f(s) \}.$$
 
Notice that
  $\bar{S}(f) = S(f) \cup - S(f)$.
Suppose $a + b \in S(f)$, so we have \mbox{$x_1 \in [0,1]$} and $x_2 = x_1 - (a + b) \in [0,1]$ with $f(x_1) = f(x_2)$.  
Let $F: \R \to \R$ be such that
$F|_{[x_1,x_2]} = f|_{[x_1,x_2]}$ and $F$ is periodic with period $a+b$.
Apply Proposition \ref{prop:prop} to $F$  to find a point $x_0$ such that
$F(x_0 - a) = F(x_0)$. 

By the remark we can choose $x_0 \in [x_1,x_2]$. 
If $x_0 - a \in [x_1,x_2]$ then we conclude $f(x_0 - a) = f(x_0)$, i.e., $a \in S(f)$.
On the other hand if $x_0 - a < x_1$ then $x_0 - a + (a+b) = x_0 + b \in [x_1,x_2]$
and so we conclude $f(x_0 +b) = f(x_0)$ and so \mbox{$-b \in \bar{S}(f)$},
and thus $b \in S(f)$.
\end{proof}

\section{At least half the lengths}

The main ingredients in this section  are Hopf's theorem and  the central symmetry $\ell \to d - \ell$ of the interval $[0,d]$ .
Let $S_d(f) := S(f) \cap [0,d]$.
The image of $S_d(f)$ by this symmetry is the set $T_d(f) := \{\ell: d-\ell \in S_d(f)\}$.
Denote the Lebesgue measure on $[0,1]$ by $\lambda$.
We have the following result:

\begin{theorem} Let $f:[0,1]\to\mathbf{R}$ be a continuous function with $f(0)=f(1)=0$. Then the length  of $S_d(f)$ satisfies 
$\lambda(S_d(f)) 
\ge \frac{1}{2}d$ for each $d \in (0,1]$. This inequality is strict if $d \not \in S(f)$. For $d=1$ this yields $$\lambda(S(f)) \ge \frac{1}{2}.$$
This last bound is sharp.
\end{theorem}

\begin{proof} Notice that $1 \in S(f)$, so $S(f) \ne \emptyset$. Fix $d \in S(f)$.
Since the map $\ell \to d- \ell$ preserves the Lebsegue measure restricted to $[0,d]$ we have 
$$\lambda(T_d(f)) = \lambda(S_d(f)).$$

By Hopf's theorem if  $a + b \not \in S^*$ then we must have either $a$ or $b$ not in $S^c$, i.e., at least one of $a$ or $b$ belongs to $S(f)$.  We apply this to $d = \ell$ and $b = d - \ell$, since  $d \in S(f)$ we conclude that either $\ell$ or $d - \ell$ is contained in  $S(f)$, or equivalently either $\ell \in S_d(f)$ or $\ell \in T_d(f)$; i.e.,
$$S_d(f) \cup T_d(f) = [0,d].$$
If $\lambda(S_d(f)) < d/2$ the two display equations yield an immediate contradiction, thus the first statement of the theorem holds for any $d \in S(f)$.

Next suppose $d \not \in S(f)$, then since $S(f)$ is closed and $1 \in S(f)$ we can find $d_m \in S(f)$ such that $(d,d_m) \subset S^*$.  But then we have 
$S_d(f) = S_{d_m}(f) \setminus \{d_m\}$ and so
$\lambda(S_d(f)) = \lambda(S_{d_m}(f)) \ge d_m/2 > d/2$, which finishes the proof of the first and second statements.

To show that the bound is sharp, we construct an example of a function whose chord set has length $1/2$, with any desired number $n$ of mountains and valleys.

For each  $n\in\mathbf{N}$, consider the complementary sets
\begin{align*}
S_n &= \left[\frac 0n, \frac 1{n+1}\right], \left[\frac 1 n, \frac{2}{n+1}\right], \ldots, \left[\frac{n-1}n,\frac n{n+1}\right], \text{ and }\\
{S_n}^* &= \left(\frac{1}{n+1},\frac{1}{n}\right), \left(\frac{2}{n+1}, \frac{2}{n}\right), \ldots \left(\frac{n-1}{n+1}, \frac{n-1}{n}\right), \left(\frac{n}{n+1}, 1\right).
\end{align*}
It is clear that ${S_n}^*$ is an additive set, because all of the intervals are integer multiples of the first one. Thus since ${S_n}^*$ is an open additive set, by Hopf's theorem it is the complement of the chord set of some function on $[0,1]$. That chord set is $S_n$.

The total length of each $S_n$ is $1/2$:

\begin{align*}
\sum_{k=1}^n \left(\frac{k}{n+1}-\frac{k-1}{n} \right)
& = \sum_{k=1}^n \frac{k}{n+1} - \sum_{k=1}^n \frac{k-1}{n} \\
& = \frac 1{n+1}\frac{n(n+1)}{2} - \frac{1}{n}\frac{(n-1)n}{2} = \frac 12.
\end{align*}

\begin{figure}[ht]
     \centering
     \includegraphics[width=0.45\textwidth]{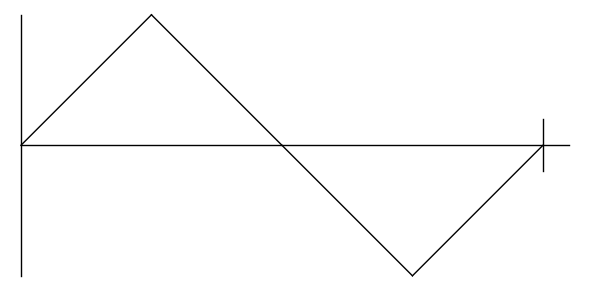} \quad
     \includegraphics[width=0.45\textwidth]{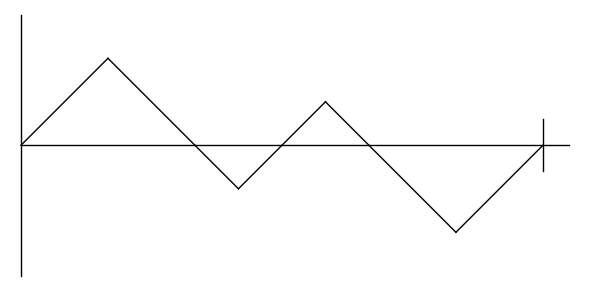} \\
     \includegraphics[width=0.45\textwidth]{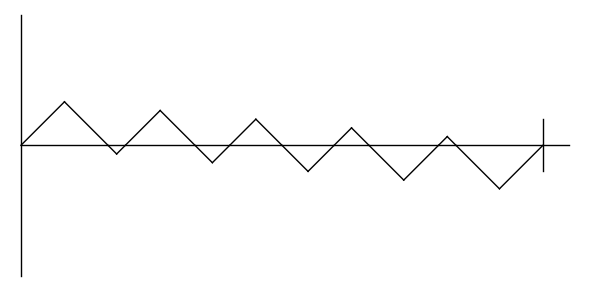} \quad
     \includegraphics[width=0.45\textwidth]{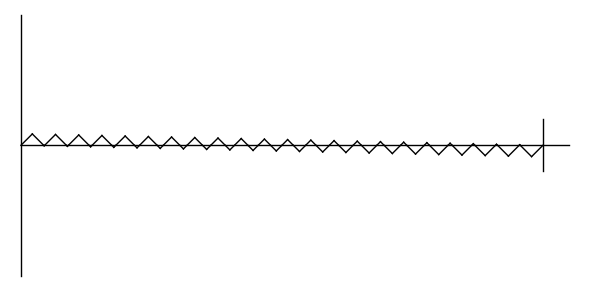}
     \caption{Functions having a horizontal chord set of length exactly $1/2$, with 1, 2, 5 and 22 mountains, respectively.}
     \label{fig:sharp-bound}
 \end{figure}

Given any open, additive set $S^*$, Hopf gives a construction for a function $h_S(x)$ whose 
horizontal chord set is exactly $S^*$. Examples of functions $h_{S_n}(x)$ with $n$ mountains and valleys, and chord set with length $1/2$, given by Hopf's construction are shown in Figure \ref{fig:sharp-bound}. Thus, for each $n$, this construction gives the desired function with $n$ bumps 
and horizontal chord set of length $1/2$.
\end{proof}

Our proof leads to the following natural question: does there exist a example of a function $f$
with countably many mountain ranges and countably many valley ranges such that $\lambda(S(f)) = 1/2$?

The following observation suggests that this is harder to achieve than one might imagine. 
Let $C$ be the middle thirds Cantor set, and let $f: [0,1] \to \mathbb{R}$ be any continuous function such that $f(x) = 0$ for all $x \in C$. The intersection of shifted Cantor sets is well studied; in particular, it is known that
\mbox{$\displaystyle  \{\ell \in [0,1]: (C - \ell) \cap C \ne \emptyset\} = [0,1]$} (see for example \cite{PP}). Since this set is just
\mbox{$\{\ell \in [0,1]: f(x - \ell) = f(x) = 0\}$}, it
is a subset of the horizontal chord set; thus  we conclude that $f$ has the full chord property, rather than the desired measure $1/2$.

\section{Acknowledgments} 
%[omitted]
We thank CIRM for an excellent working environment during our \guillemotleft~recherche en binomes~\guillemotright\  collaboration in July 2022. We used Desmos \cite{D} to experiment, explore all of the many cases, and create our figures. We thank the referees, whose suggestions improved our article.

\appendix
\section{Two mountain ranges separated by a valley range}

It would be nice to give a classification of which functions' chord sets have the full chord property, but the general case seems complicated. 
Thus here we will only analyze the simplest non-trivial case, the piecewise affine case with two mountains and one valley in between.

Let $w_\ell,w_v,w_r > 0$ be the widths of the left mountain, valley, and right mountain,  which are normalized to  add up to one, and let $h_\ell,h_r >0$  and $h_v < 0$ be their heights. 
Since the chord set does not depend on vertical scaling we normalize $\max(h_\ell,h_r) =1$.
We also use the widths of the ascent and descent of each mountain, which we call $a_\ell, d_\ell, a_r$ and $d_r$, respectively (see Figure \ref{fig:mountains-intersect}). 

\begin{figure}[!ht]
    \centering
    \includegraphics[width=0.5\textwidth] 
    {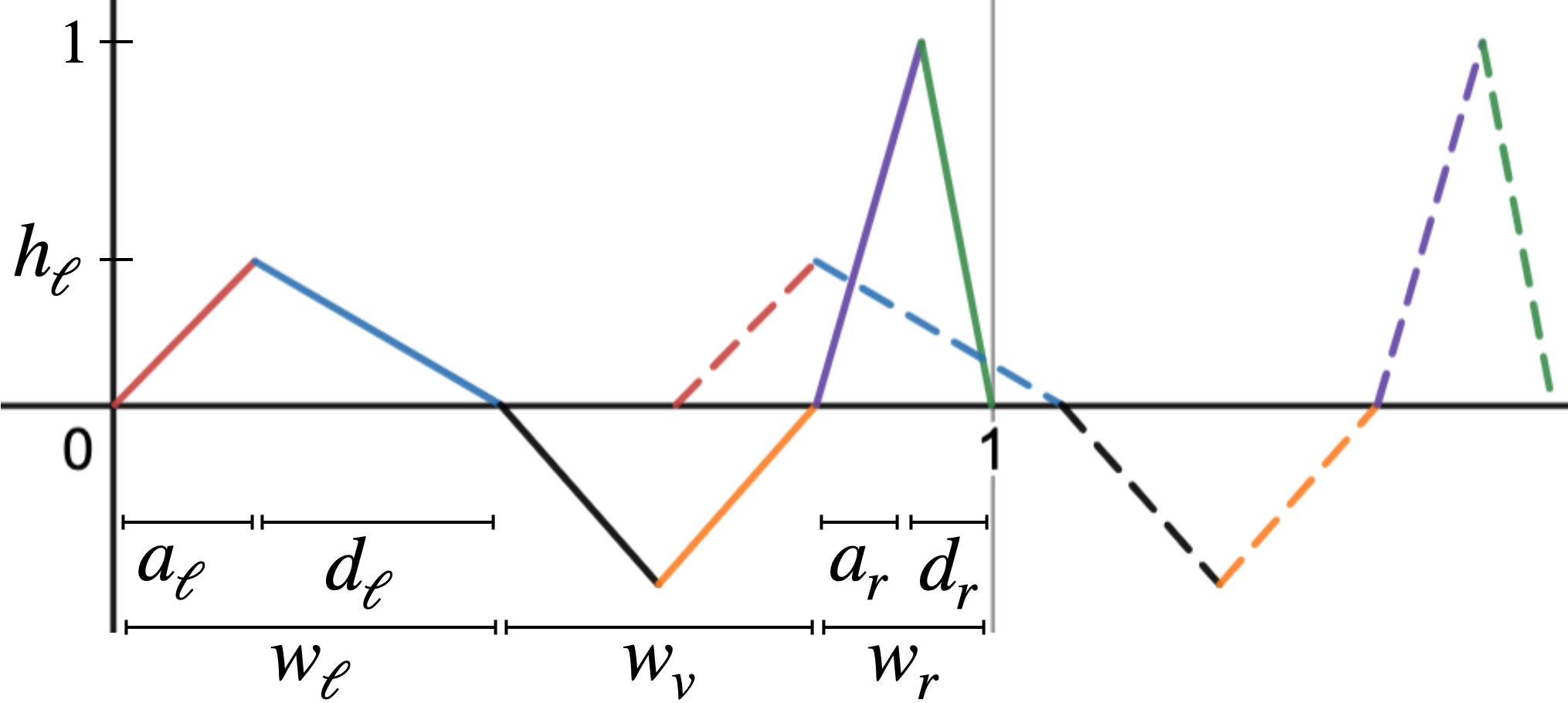}
    \caption{A narrow tall mountain always intersects a wide short mountain if their bases intersect. The image also shows the variables defined above.}
    \label{fig:mountains-intersect}
\end{figure}
 
 We will use the notation $f_{\w,\a,\d,\h}$. Certain  variables are redundant since a width is the sum of its ascent and descent, and one of the ascents or descents or widths is also redundant because of our normalization, thus we have a seven dimensional parameter space along with one discrete parameter.
However the constraints in the next theorem live in a 5-dimensional subspace of our 7-dimensional space because of the three parameters needed to describe a valley only its width  plays a role.

\begin{theorem}
The piecewise linear map $f_{\w,\d,\a,\h}$ does {\bf not} have the full chord property if and only if the parameters satisfy one of the following two conditions 

\begin{enumerate}
\item  (i) $w_\ell < w_r$, (ii) $h_\ell < 1$, (iii)
${h_\ell \cdot a_r}+ 1 - w_r -a_\ell < w_v + w_r $ \ and (iv)
$1 - h_\ell \cdot d_r-a_{\ell} > 
\max(w_r;w_\ell+w_v)$, or

\item (i) $w_r < w_\ell$, (ii) $h_r < 1$, (iii)   
${h_r \cdot a_\ell}+ 1 - w_\ell-a_r < w_v + w_\ell $ \ and (iv) 
$1 - h_r \cdot d_\ell-a_r >  \max(w_\ell;w_r+w_v)$.
\end{enumerate}
\end{theorem}

\begin{proof}
There are two different ways to get a chord value.
The first possibility is that  either the chord connects the two sides of the same mountain or the two side of  the same valley, for such chord lengths Theorem \ref{mountainsarefull} implies $S_{\rm init} := [0,w_{\rm max}] \subset S(f)$ where $w_{\rm max} := \max(w_\ell,w_v,w_r)$.

Because of the special form of our graph there is only one other type of chord length possible, namely the lengths of chords that arise between the two mountains, or equivalently when the left mountain  shifted by such a length intersects the right mountain. 

Clearly if the narrower mountain is at least as high 
as the wider mountain, then the mountains always intersect when their bases overlap (Figure \ref{fig:mountains-intersect}), so such a function has the full chord property. Thus in order not to have the full chord property, one mountain must be narrower and shorter. We can restrict to two cases:

\noindent {\bf  Case (1)} (Figure \ref{fig:wide-intersects}) Suppose first that the narrower and shorter mountain is on the left
((i) $w_\ell < w_r$ and (ii) $h_\ell < h_r = 1$). 

\begin{figure}[!ht]
    \centering
\includegraphics[width=0.8\textwidth]{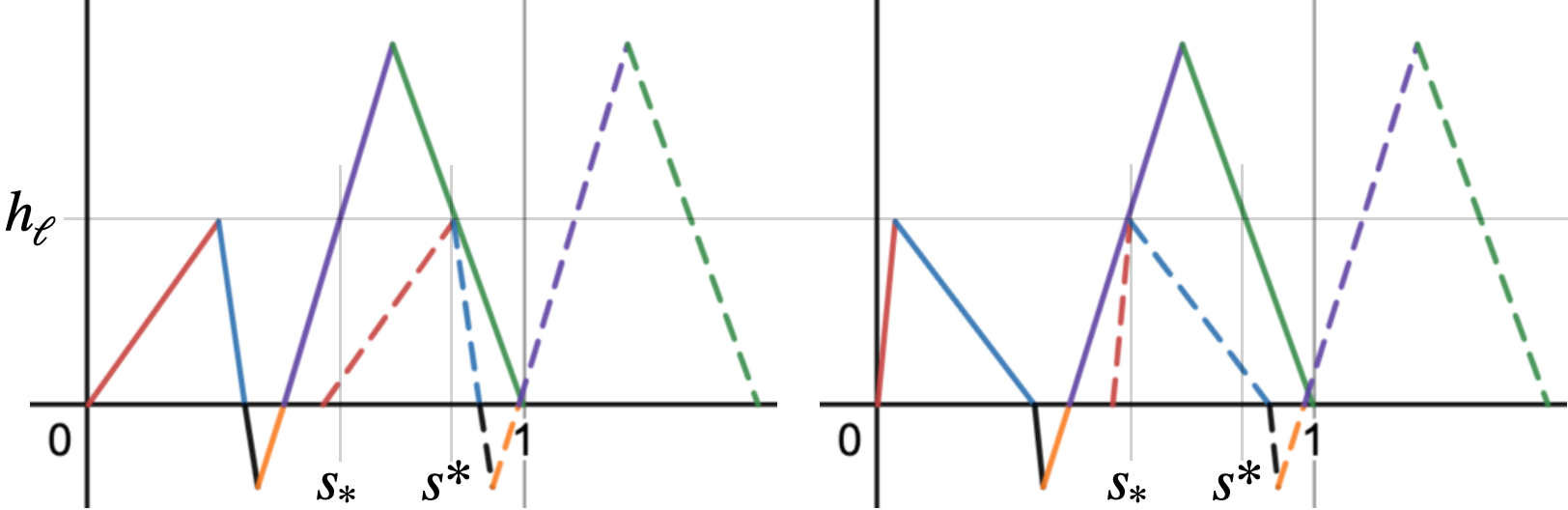}
    \caption{The case when a narrower \it{and} shorter mountain is on the left, along with the definitions of $s_*$ and $s^*$.}
    \label{fig:wide-intersects}
\end{figure}

No matter which parameters we choose, chord lengths between the two mountains  arise if the bases of the  shifted left mountain and the right mountain intersect,  but neither base is contained in the other base. 
This set consists of two intervals described below.

 When a parameter value belongs to the interval \mbox{$S_{\rm mid} := [w_v, w_\ell + w_v]$}, the shift of the left mountain  intersects
the ascent of the right mountain, 
and so $S_{\rm mid} \subset S(f)$ (Figure \ref{fig:S_mid}).

When a parameter value belongs to the interval  \mbox{$S_{\rm fin} := [w_r + w_v,1]$}, the shift of the left mountain  intersects
the descent of the right mountain, so $S_{\rm fin} \subset S(f)$.

\begin{figure}[ht]
    \centering
    \includegraphics[width=\textwidth]{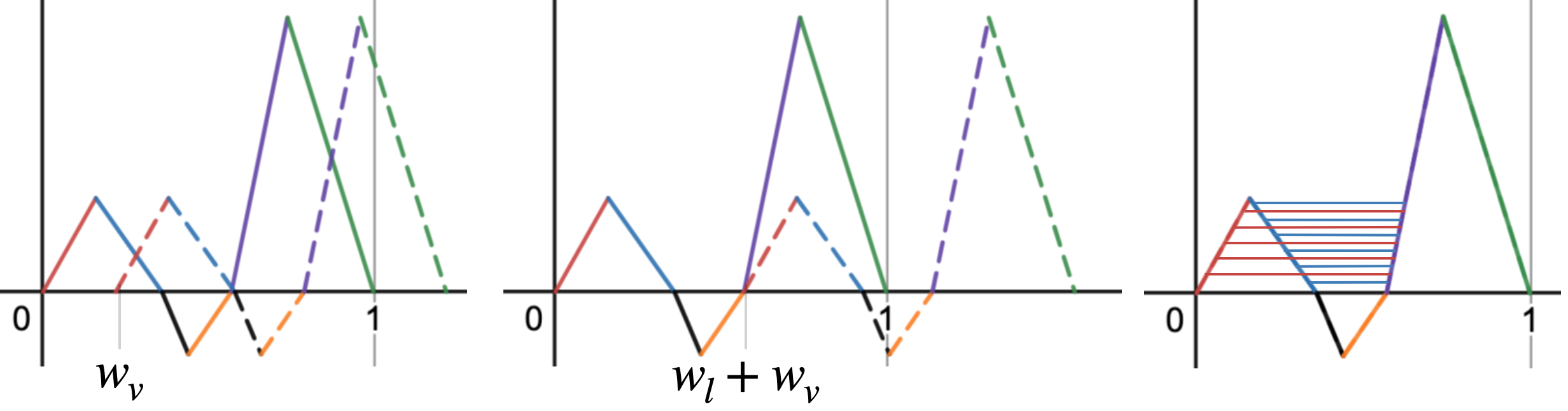}
    \caption{(Left) a shift by the parameter $w_v$ (middle) a shift by the parameter $w_l+w_v$ (right) the horizontal chord lengths in $S_{\rm mid}$}
    \label{fig:S_mid}
\end{figure}

Summarizing, we have shown that $ S_{\rm init} \cup S_{\rm mid} \cup S_{\rm fin} \subset S(f)$ no matter which  parameter values we choose, as long as the narrower shorter mountain is on the left.  

We will now analyze the complementary set $T := (S_{\rm init} \cup S_{\rm mid} \cup S_{\rm fin})^c$. 
The intervals $S_{\rm init}$ and $S_{\rm mid}$ abut or overlap since $w_v \le w_{\rm max}$.
The larger of the two right endpoints of these two intervals is
$$ \max(w_{\rm max},w_\ell + w_v) = \max(w_\ell,w_v,w_r, w_\ell + w_v) = 
\max(w_r, w_\ell + w_v).$$
Thus we have
\begin{equation} \label{T}
T=(\max(w_r;w_\ell+w_v), w_v + w_r).
\end{equation}
Note that the set $T$ is contained in the set of parameters for which the base of the shifted left mountain is inside the base of the right mountain, and is a strict subset if $S_\text{mid} \subset S_{\text{init}} \setminus [0,w_v)$, or equivalently if $w_r > w_\ell + w_v$.

If the ascent of the smaller mountain is steeper than the ascent of the larger mountain, or if the descent of the smaller mountain is steeper than the descent of the larger mountain, then we get an additional interval of chord values in $S(f)$ that are in $T$, one for each such case.

There are three solutions to the equation $f(s) = h_{\ell}$, $a_\ell < s_* < s^*$ 
(pictured in Figure \ref{fig:wide-intersects}).  
The zero, one or two additional intervals of chord values in $T$ are then 
$T_1 :=[\max(w_r;w_\ell+w_v);s_*-a_\ell]$
and
$T_2 :=[s^*-a_{\ell};w_v + w_\ell ]$,
where here an interval $[a,b]$ with $b<a$ is interpreted as the empty interval.

To calculate $s_*$ and $s^*$ consider the equations of the ascent and descent of the right mountain, recalling that $h_r=1$:
  \begin{equation*}
  y = \frac{1}{a_r}(s - (1-w_r)) \quad \hbox{ and } \quad
 y =  \frac{-1}{d_r}(s - 1).
\end{equation*}
Replacing $y$ by $h_{\ell}$ in these two equations yields 
\begin{equation*}
s_* = {h_\ell \cdot a_r}+ 1 - w_r \hbox{  and  } s^* =
1 - h_\ell \cdot d_r.
\end{equation*}
Thus the two additional intervals are
\begin{eqnarray}\label{T1}
T_1 &:=& [\max(w_r;w_\ell+w_v);{h_\ell \cdot a_r}+ 1 - w_r -a_\ell] \text{  and } \\
T_2 &:=&[1 - h_\ell \cdot d_r-a_{\ell};w_v + w_\ell ]. \label{T2}
\end{eqnarray}

The intervals $T_1$ and $T_2$ never overlap since $s_* < s^*$, 
but either one of them can cover $T$. Thus,  the full chord property in the case under consideration
is equivalent to having either $T \subset T_1$ or
$T \subset T_2$.

Comparing equations \eqref{T} and \eqref{T1}--\eqref{T2}, we see that $T_1$ covers $T$ if and only if
$${h_\ell \cdot a_r}+ 1 - w_r -a_\ell \ge w_v + w_r $$
and the right end point of $T_2$ is larger than the right endpoint of $T$, thus $T_2$ covers $T$ if and only if
$$1 - h_\ell \cdot d_r-a_{\ell} \le 
\max(w_r;w_\ell+w_v).$$
If neither of these last two equations, which correspond to the negations of (iii) and (iv), hold, then we do not have the full chord property.

\noindent {\bf Case (2)} (Figure  \ref{fig:right-smaller}) Now suppose that the narrower and shorter mountain is on the right 
((i) $w_r < w_\ell$ and (ii) $h_r < h_\ell = 1$).  

\begin{figure}[ht]
    \centering
    \includegraphics[width=0.6\textwidth]{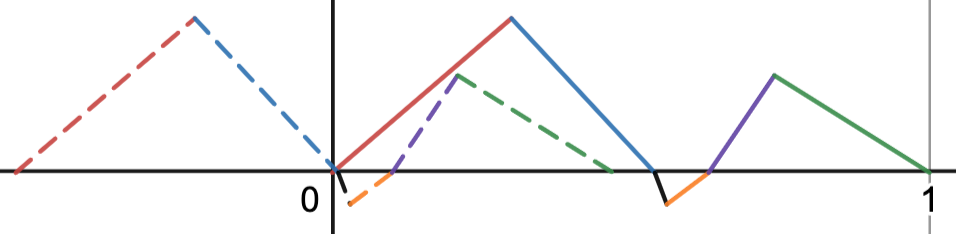}
    \caption{The case when the right mountain is narrower and shorter}
    \label{fig:right-smaller}
\end{figure}

We take the walk backwards, in other words, we consider 
the function \mbox{$g(s) :=  f\circ i (s)$} where 
$i$ is the isometry   \mbox{$i: s \mapsto 1-s$}.
The function $g$
has the same chord set as $f$, and furthermore if $f$ is in case (2) then $g$ is  in case (1). 
Applying the characterization proven in case (1) to $g$
and reinterpreting this characterization in terms of the parameters of  $f$ yields:
(i) $w_r < w_\ell$, (ii) $h_r < 1$, (iii) 
${h_r \cdot a_\ell}+ 1 - w_\ell-a_r < w_v + w_\ell $, and
(iv) $1 - h_r \cdot d_\ell-a_r > 
\max(w_\ell;w_r+w_v).$
\end{proof}

A Monte Carlo simulation suggests that the percentage of such functions with the full chord property is approximately $70.4\%$. In principle an exact calculation is possible, but as the constraints are not linear, it is not elementary.

\begin{thebibliography}{99}
\bibitem{BDD} Keith Burns, Orit Davidovich, Diana Davis, {\it Average pace and horizontal chords}, The Mathematical Intelligencer, 39(4) (2017) 41-45.

\bibitem{D} {\it Desmos}, \url{https://www.desmos.com/calculator/}, accessed 2022--2023.

\bibitem{GPY} Jacob E.\  Goodman, Janos Pach and Chee K.\  Yap {\it Mountain Climbing, Ladder Moving, and the Ring-Width of a Polygon},  
The American Mathematical Monthly,  96 (6) (1989) 494--510.

\bibitem{H} Heinz Hopf, {\it \"Uber die Sehnen ebener Kontinuen und 
die Schleifen geschlossener Wege}, Comment.~Math.~Helv., 9 (1936) 303--319.

\bibitem{K} Tam\'as Keleti,
{\it The mountain climbers' problem }, Proc.~AMS, 117(1) (1993) 89--97.


\bibitem{L} Paul L\'evy, {\it Sur une G\'en\'eralisation du Th\'eor\'eme de Rolle}, C.R.~Acad.~Sci., Paris, 198 (1934) 424--425.

\bibitem{PP}
Steen Pedersen,  Jason D.~Phillips, {\it On intersections of Cantor sets: Hausdorff measure}, Opuscula Math. 33(3) (2013)  575--598.

\bibitem{W}
 James Whittaker, {\it A mountain-climbing problem}, Canad. J.\ Math 18 (1966) 873--882.

\end{thebibliography}
\end{document}